\documentclass[11pt]{amsart}
\usepackage{cite}
\usepackage{geometry}
\usepackage{amsmath,amssymb,graphicx,blkarray,amsthm} 
\usepackage{mathtools}
\graphicspath{ {\string~/Desktop/} }

\usepackage{bbm}
\usepackage{color,subfig}
\usepackage[table]{xcolor}
\usepackage[colorlinks,linkcolor=blue,citecolor=blue,urlcolor=black,bookmarks=false,hypertexnames=true]{hyperref} 
\usepackage{float}
\usepackage{enumitem}
\usepackage{subdepth} 

\usepackage{tikz}
\usetikzlibrary{calc,angles,positioning,quotes,decorations, automata, arrows}
\usepackage{tkz-euclide}

\usepackage{verbatim}
\usepackage{tikz-cd}

\usepackage[all,cmtip]{xy}

\newcommand{\Per}{\mbox{Per}}
\newcommand{\Aper}{\mbox{Aper}}
\newcommand{\Image}{\mbox{Im}}
\newcommand{\bfu}{\textit{\textbf{u}}}
\newcommand{\bfx}{\textit{\textbf{x}}}
\newcommand{\bfv}{\textit{\textbf{v}}}

\theoremstyle{plain}
\newtheorem{theorem}{Theorem}
\newtheorem{proposition}[theorem]{Proposition}
\newtheorem{lemma}[theorem]{Lemma} 
\newtheorem{corollary}[theorem]{Corollary}

\theoremstyle{definition}
\newtheorem{definition}[theorem]{Definition}

\newtheorem{example}[theorem]{Example}

\newcommand{\Z}{\mathbb Z}
\newcommand{\N}{\mathbb N}
\newcommand{\I}{\mathbb I}

\DeclareMathOperator{\id}{id}

\title{Semicocycle discontinuities for substitutions and reverse-reading automata}
\author{Gandhar Joshi }
\address{School of Mathematics and Statistics, The Open University, U.K.}
\email{gandhar.joshi@open.ac.uk}
\author{Reem Yassawi}
\address{   School of Mathematical Sciences, Queen Mary University of London, UK
}
\email{r.yassawi@qmul.ac.uk}
\thanks{This work was supported by  the
   EPSRC  grant number EP/V007459/1.}

\begin{document}
	\maketitle

\begin{abstract}
In this article we define the   semigroup associated to a primitive substitution. We use it 
to
 construct a minimal automaton which generates a substitution sequence ${\bfu}$ in reverse reading. We show, in the case where the substitution has a coincidence, that  this automaton completely describes the {\em semicocycle discontinuities} of ${\bfu}$.
\end{abstract}

	{\em In Memoriam}: Uwe Grimm died unexpectedly  in October 2021. He was the first author's supervisor and the second author's close colleague and dear friend. His loss is deeply felt by both; we  dedicate  this paper to his memory.
	
	\section{Introduction}
		Automatic sequences are codings of constant length-$\ell$ substitutional fixed points. The latter generate discrete dynamical systems whose study is  ubiquitous in the literature. Cobham \cite{Cobham} showed that an equivalent definition of an automatic sequence is that it is a 
	 sequence generated by a {\em finite automaton with output} (Definition \ref{DFA, DFAO}).  To generate $u_n$, the $n$-th entry of the sequence, one feeds the base-$\ell$ expansion  $(n)_\ell$ of $n$ into the automaton, to arrive at a final state, and $u_n$ is the coding, or  output, of the final state.  One has to set the order of reading of $(n)_\ell$, i.e., whether one reads starting with the most significant digit, in which case the reading is called {\em direct}, or starting with the least significant digit, in which case the reading is called  {\em reverse}.
	A fixed automatic sequence can be generated in either direct or reverse reading, but in general, generating automata will be different. See Section \ref{sec:preliminaries} for definitions and background.
	
	The relevance of the (minimal) direct-reading automaton in dynamics has long been understood. Indeed, Cobham's proof tells us that we can define this automaton directly in terms of the substitution rule. Furthermore, this direct-reading automaton also enables us to locate the irregular fibres of the maximal equicontinuous factor \cite{CQY}, and this is useful as these fibres  drive the interesting dynamics in the shift. Cobham defined automatic sequences in terms of direct reading automata, and he did not consider reverse-reading  automata. Eilenberg identified that reverse-reading automata gave information about the {\em $\ell$-kernel} of a sequence $\bfu$, that is, the set of all subsequences of $\bfu$ whose indices run along an arithmetic progression and such that its difference is a power of $\ell$.
Furthermore, the cardinality of the $\ell$-kernel is also the size of a minimal automaton that generates $\bfu$ in reverse reading (Theorem \ref{Eilenberg}). The reverse-reading automaton has been used extensively to generate algebraic characterisations of automatic sequences, in particular when their entries belong to a field, \cite{Christol} or even sometimes a ring \cite{Denef-Lipshitz}. However there has been little or no use of the reverse-reading automaton in dynamics.
	
	This paper arose out of a desire to understand the dynamical meaning of the $\ell$-kernel, and also to understand what the reverse-reading automaton tells us about the dynamical system generated by $\bfu$. For a special class of shift-dynamical systems, namely the {\em Toeplitz shifts} (see Section \ref{sec:odometer}), the $\ell$-kernel  has a familiar dynamical interpretation.  By Theorem \ref{thm:Toeplitz}, Toeplitz shifts are guaranteed to contain at least one {\em Toeplitz  sequence},  i.e.,  a sequence which   is constructed by filling in entries of   one arithmetic progression of indices at a time, with a constant symbol, in a way such that eventually the sequence is completely specified.  The first such sequences were studied by 
	Oxtoby \cite{Oxtoby}, who defined a  Toeplitz sequence that generates a minimal shift supporting two invariant measures. 
	
	At the other end, Jacobs and Keane \cite{Jacobs-Keane},  introduced what is possibly the simplest nontrivial Toeplitz sequence, the well-known one-sided {\em period-doubling sequence}. It  is a substitutional fixed point, and  generates a uniquely ergodic shift which has very little ``irregularity" as follows.	Let $(X,\sigma)$ be a  Toeplitz shift, where $X\subset \mathcal A^{\I}$ with $\mathcal A$ a finite set, and where $\I$ equals $\Z$ or $\N_0$. Not all elements in $X$ are Toeplitz sequences.  How far away a sequence $(x_n)_{n\in \I}$ is from being Toeplitz is a function of the cardinality of  the set of discontinuities  of the  {\em semicocycle} that it defines; see Definition \ref{def:semicocycle}. 	 The discontinuities of the semicocycle are strongly  linked to the values that it  takes on arithmetic progressions; see  Lemma \ref{lem:semicocycle}. 
   In Example \ref{pd_ex} we show that
 the one-sided period-doubling sequence defines a  continuous semicocycle on $\mathbb N_0$, but any of its extensions to a two-sided sequence defines a semicocycle on $\Z$ with only one discontinuity, at $n=-1$. In general, Toeplitz shifts will contain non-Toeplitz elements whose semicocycle has far more discontinuities; see  Example \ref{bigdiag_rrex}.

	Williams \cite{Williams} carried out the first systematic  study of  Toeplitz shifts. Given a point $\bfx= (x_n)$ in $X$, she partitions its indices up into two parts, the {\em periodic part} $\Per(\bfx)$, and  the {\em aperiodic part} $\Aper(\bfx)$ (see Section \ref{Toeplitz-sequence}).  The periodic part is the set of indices that have been filled, with a constant symbol, using arithmetic progressions. The aperiodic part is what is left over. Thus, if $\bfx$ is actually a Toeplitz sequence, then $\Aper(\bfx)$ is empty, and the semicocycle is continuous. The existence of a large $\Aper(\bfx)$ is what makes the dynamics more complex, e.g., it can lead to the existence of several invariant measures, measurable but not continuous eigenvalues, positive entropy, and so on.
	
	Our main result is Theorem \ref{thm:main-result}. It tells us  that given an automatic sequence $\bfu$ which belongs to a Toeplitz shift, the reverse-reading automaton which generates $\bfu$ gives us complete information about the indices which belong to  $\Aper(\bfu)$, and hence the set of discontinuities of the semicocycle  defined by $\bfu$. To prove Theorem \ref{thm:main-result}, we first construct a suitable automaton that generates $\bfu$ in reverse reading, in Theorem \ref{thm:automaton}. This automaton is minimal, so it can also be obtained using Eilenberg's classic construction, but we label the states in a different manner using the {\em  semigroup} of the corresponding substitution (Definition \ref{def:structure-semigroup}). 
	
	Finally, consider a formal power series $f(x)$ whose coefficients form an automatic sequence. Reverse-reading automata are instrumental in expressing $f(x)$ as either a root of a polynomial, as given by Christol's theorem, or as a diagonal of a higher dimensional rational function, as given by Furstenberg's theorem; for an exposition of either result, see \cite{allouche-shallit-2003}. In future work, we will use the reverse-reading automaton that we have constructed to translate certain dynamical features of substitution shifts to algebraic properties of their annihilating polynomials.

	\section{Preliminaries}\label{sec:preliminaries}
	
    Let $\N$ and $\N_0$ denote the positive integers and the  non-negative integers respectively. Given $n\in \N$, there is a unique expression $n=\sum_{i=0}^{j-1} n_i\ell^i$, where each $n_i\in \{0,\dots , \ell -1\}$ and $n_{j-1}\neq 0$; we call this the {\em canonical  base-$\ell$ expansion of $n$} and we denote it as $(n)_\ell:= n_{j-1}\dots  n_0\cdot$. Note the use of the radix point ``$\cdot$", which indicates that $n_0$ is the least significant digit. We define the canonical  base-$\ell$ expansion of $0$ to be $(0)_\ell:=0$.
    Negative integers  also have a base-$\ell$ expansion, which we define as follows. If the natural number $n$ satisfies $(n)_\ell=  n_{j-1} \dots  n_0\cdot $, then let $w\in \{0,\dots , \ell-1\}^{ j}$ denote the unique word of length $j$ such that \[(n)_\ell+w\cdot = 1\overbrace{0 \dots 0  }^{j}\cdot \, ,\]
    where addition is performed mod $\ell$ and with carry   from right to left. With this, we define the canonical base-$\ell$ expansion of the negative integer  $-n$ as the sequence
    \[   (-n)_\ell:=      \dots (\ell -1) (\ell-1) w\cdot \, ,\]
  the reason being that in the ring of $\ell$-adic integers  $\Z_\ell $, $(n)_\ell+ (-n)_\ell $ equal the zero element $\dots 000\cdot$; this further explains the use of radix point.  See Section \ref{sec:odometer} for the construction of $\Z_\ell$.  Given a digit $k\in \{0,1,\dots , \ell -1\}$, let $\bar k$ denote the constant sequence $\dots kkk$.  With this notation, $(-n)_\ell= \overline{(\ell -1)}\,w\cdot$.

    For example, if $\ell=3$, and $n=25$, then $(25)_3=221\cdot$, $w=002$, and $(-25)_\ell:= \bar{2}002 \cdot $.

	Let $\mathcal A$ be a finite alphabet. In this article, the indexing set $\I $ equals  $\N_0$ or $\Z$. An element in $\mathcal A^\I$  is written $\bfu=\left(u_n\right)_{n \in \I}$.
	
Next we define some finite-state automata with which we will work. We gently modify existing definitions in \cite{allouche-shallit-2003}  to allow  the generation of two-sided sequences (see Definition \ref{def:automatic}).

\begin{definition}\label{DFA, DFAO}
A {\em deterministic finite automaton } (DFA)  is a $4$-tuple $\mathcal M=(\mathcal S, \Sigma, \delta, \mathcal S_0)$, where $\mathcal S$ is a finite set of {\em states}, $ \mathcal S_0\subseteq \mathcal S$ is a set of initial states, $\Sigma$ is a finite  {\em input} alphabet,  and $\delta:\mathcal S\times \Sigma \rightarrow \mathcal S$ is the {\em transition} function.

A {\em deterministic finite automaton with output} (DFAO)  is a $6$-tuple $\mathcal M=(\mathcal S, \Sigma, \delta, \mathcal S_0, \mathcal A, \Omega_0)$, where 
$(\mathcal S, \Sigma, \delta, \mathcal S_0)$ is a DFA,
$\mathcal A$ is a finite {\em output} alphabet, and  where for each initial state $s_0\in \mathcal S_0$, there is  an associated  {\em output function} $\omega_0\in \Omega_0$,
$\omega_0:\mathcal S\rightarrow \mathcal A$. 
\end{definition}

For example, the DFAO described in Figure \ref{pd_dir-pic} has $\mathcal S=\{a,b,c \}$,  $\Sigma=\{0,1,2,3,4,5,6,7,8\}$,  $\mathcal S_0= \{a,b\}$, $\delta:{\mathcal S}\times\Sigma\to{\mathcal S}$ as given in Table \ref{tab:deltatable}, the output alphabet ${\mathcal A}={\mathcal S}$, and the output functions $\omega_a$ and $\omega_b$ are the identity map.
\begin{table}[ht]
        \centering
        \begin{tabular}{|c||c|c|c|c|c|c|c|c|c|}
        \hline
             $\delta$ &$0$& $1$&$2$&$3$& $4$&$5$&$6$& $7$&$8$\\
             \hline\hline
             $a$& $a$ &$c$&$b$& $b$ &$b$&$a$& $b$ &$a$&$a$\\
             \hline
             $b$& $b$ &$a$&$a$& $a$ &$c$&$b$& $a$ &$c$&$b$\\
             \hline
             $c$& $b$ &$a$&$a$&$b$ &$a$&$a$& $a$ &$c$&$b$\\
             \hline
        \end{tabular}
        \caption{The map $\delta$ for Figure \ref{pd_dir-pic}.}
        \label{tab:deltatable}
\end{table}

The function $\delta$ extends in a natural way to the domain $\mathcal S \times \Sigma^*$, where $\Sigma^*$ is the set of all finite words on the alphabet $\Sigma$.
One way to do this is to  define $\delta(s, n_k \cdots n_1 n_0) := \delta(\delta(s, n_0), n_k \cdots n_1)$ recursively. 
Here, the
way we have written things, we first feed $n_0$, then $n_1$, etc, so that we start with the least significant digit, the automaton is then called a \emph{reverse-reading automaton}.
 However, we can also extend $\delta $ to words by starting
with   the most significant digit  $n_k$,  and ending with $n_0$. In other words we can also define $\delta(s, n_k \cdots n_1 n_0) := \delta(\delta(s, n_k), n_{k -1} \cdots n_0)$, then the automaton is a \emph{direct-reading automaton}.  
  In this article the automata that we investigate are reverse reading automata.

   If $n\in -\N$, recall that its  expansion $(n)_\ell = \overline{(\ell -1)}\,n_j\cdots n_0\cdot$ is an infinite sequence which must eventually equal $\overline{(\ell -1)}\cdot$.
  For negative integers, instead of working with the canonical base-$\ell$ expansion,   we shall see that we can set things up to work with  the finite  word $(\ell -1)\,n_{j}\cdots  n_0\cdot$.
  This is because we will feed the transition function $\delta$ with base-$\ell$ expansions of integers, and the function $\delta$  only accepts finite words from $\Sigma^*$. We will arrange  it so that a ``finite" expansion of a negative integer is sufficient for our needs, by ensuring that the appropriate  edge labelled $\ell -1$ is a loop.

  \begin{definition}\label{def:automatic}
A sequence $\left(u_n\right)_{n \geq 0}$  in $\mathcal A^{\N_0}$ is {\em $\ell$-automatic} if there is  a DFAO
$(\mathcal S, \{0, \dots,\ell-1\}, \delta, \{s_0\}, \mathcal A, \{\omega_0\})$ such that $u_n = \omega_0(\delta (s_0, n_k \cdots n_0))$ for $(n)_\ell= n_k\cdots n_0\cdot$.
Similarly, a sequence $\left(u_n\right)_{n \in \Z}$  in $\mathcal A^{\Z}$ is {\em $\ell$-automatic} if there is a DFAO 
$(\mathcal S, \{0, \dots, \ell - 1\}, \delta, \{s_0, s_1\}, \mathcal A, \{ \omega_0, \omega_1\})$ such that  $(\mathcal S, \{0, \dots, \ell - 1\}, \delta, \{s_0\}, \mathcal A, \{\omega_0\})$  generates $\left(u_n\right)_{n\geq 0}$ starting at $s_0$, and such that  if $n<0$, then
$u_n = \omega_1(\delta (s_1, (\ell -1)\,n_k \cdots n_0))$ for $(n)_\ell=  \overline{(\ell-1)}\,n_k\cdots n_0\cdot$.
\end{definition}

Note that in general a sequence generated by a DFAO  in reverse reading will generally not equal the sequence generated by the same DFAO  in direct reading. Also, an automatic sequence generated by a DFAO in direct reading can be generated by a, possibly different, DFAO in reverse reading, and vice versa. To transition from a direct-reading to a reading-reverse automaton one first reverses all edges in the given automaton, see for example \cite[Theorem 4.3.3]{allouche-shallit-2003}. The resulting automaton may be non-deterministic, i.e., there could be multiple edges labelled $i$ emanating from one state, and in this case the automaton needs to be determinised \cite[Theorem 4.1.3 ]{allouche-shallit-2003}.  However reversing the automaton reading  can blow up the state size of the automaton exponentially in $\ell$; see  for example the difference in state size in Figures \ref{pd_dir-pic} and \ref{fig:my_label} which are the direct and reverse reading automata that generate the same sequence in those examples.

Given a finite alphabet $\mathcal A$, let $\mathcal A^{*}$ denote the monoid of words on $\mathcal A$, and $ \mathcal{A}^+ \subset \mathcal A^{*}   $ denote the set of nonempty words on $\mathcal A$.
Let $\theta:\, \mathcal{A}^*\to \mathcal{A}^+$ be a morphism, also  called a \emph{substitution}. Note that  the image of any letter is a non-empty word. We will abuse the notation and write  $\theta:\, \mathcal{A} \to \mathcal{A}^+$.  
Let $\mathbb I= \mathbb N_0$  or $\mathbb Z$, according to whether  we are to construct a one-sided or two-sided sequence. Using concatenation, we extend $\theta$ to~$\mathcal{A}^\mathbb{I}$.
The finiteness of $\mathcal A$ guarantees that  $\theta$-periodic points, i.e., points $\bfx$ such that $\theta^k(\bfx)=\bfx$ for some  integer $k\geq 1$, exist. A substitution is  \emph{primitive} if there exists an integer $n\geq 1$ such that for all $a,b\in \mathcal A$, the letter $a$ occurs in $\theta^n(b)$.  A substitution $\theta$ is {\em aperiodic} if its $\theta$-periodic points are not shift-periodic. It is (constant) length-$\ell$ if $|\theta(a)|= \ell$ for each letter $a\in \mathcal A$.
In applications to dynamics we assume that $\theta$ is  primitive and aperiodic, but our automata results do not require this constraint. 

 The following theorem by Cobham \cite{Cobham} connects automatic sequences to one-sided, right-infinite substitution fixed sequences. A {\em coding} is a map $\omega:\mathcal A\rightarrow \mathcal B$; it extends to a map $\tau:\mathcal A^{\I}\rightarrow \mathcal B^{\I}$.

\begin{theorem}[Cobham's theorem]
A sequence $\bfu=\left(u_n\right)_{n\geq 0}$ is $\ell$-automatic in direct reading using  if and only if $\bfu$ is the image under a coding of
a right-infinite fixed point of a substitution of length $\ell$.
\end{theorem}
Furthermore, Cobham's theorem gives us a constructive way of moving between automatic sequences and  codings of substitutional fixed points.

Let $\theta$ be a length-$\ell$ substitution. We write 
$\theta (a)= \theta_0(a) \cdots \theta_{\ell-1}(a)$; 
with this notation we 
see that for each  $0\leq i \leq \ell-1$, we have a map
$\theta_i:\mathcal A \rightarrow \mathcal A$ where, for each $a\in\mathcal A$, $\theta_i(a)$ is the 
$i$-th letter of the word $\theta(a)$. 
The map $\theta_i$ can be visualised as the $i$-th column of letters if we stack the substitution words in an $|\mathcal A| \times \ell$-array. 

We will assume that our substitution is in {\em simplified form}, i.e. $\theta_0$ and $\theta_{\ell -1}$ are each idempotents in $\mathcal A^{\mathcal A}$, i.e, $\theta_0^2=\theta_0$ and $\theta_{\ell -1}^2=\theta_{\ell -1}$.
The benefit of working with idempotent first and last columns is as follows. Given a letter $a\in \mathcal A$, the sequence of words $(\theta^n(a))_{n\geq 0}$ converges to a right-infinite $\theta$-periodic sequence, i.e.,  a sequence ${\bfu}$  such that $\theta^k({\bfu})={\bfu}$. However it can happen that for small values of $n$, the word $\theta^{(n+1)k}(a)$ is not an extension of $\theta^{nk}(a)$, i.e., we have to wait a little until convergence kicks in.  The assumption that $\theta_0$ is an idempotent guarantees both that $\theta^{(n+1)k}(a)$ extends $\theta^{nk}(a)$ for each $n\geq 1$, and also that $k=1$, so that we only need concern ourselves with $\theta$-fixed points.
 Similarly, the assumption that $\theta_{\ell -1}$ is an idempotent guarantees fast convergence to 
  left-infinite $\theta$-fixed points.  The assumption that a substitution is simplified makes for simpler algebraic descriptions of a substitution, eg, see \cite{Kellendonk-Yassawi}. The  fact that $\mathcal A$ is finite and the pigeonhole principle imply that there always exists a power of $\theta$ which is in simplified form. Finally, if $\theta_0$ and  $\theta_{\ell -1}$ are idempotents, 
  then we will be able to  reduce the  base-$\ell$ expansion  $\overline{(\ell -1)} w\cdot$ of a negative integer $n$ to the finite  expansion $(\ell -1) w \cdot$, and merely feed this finite expansion in to 
  the automaton to generate $u_n$.  Similarly if $n$ is a natural number,  the same entry $u_n$ is obtained if we pad $(n)_\ell$ with leading zeros. Also, we can then lightly strengthen the statement of Cobham's theorem, so that it concerns bi-infinite fixed points of substitutions; this is desired as we work with two-sided and hence invertible shifts. From a dynamical perspective, as we work with primitive aperiodic substitutions, the dynamical system generated by $\theta$ equals that generated by any of its powers, so there is no loss of generality in assuming that $\theta$ is in simplified form.

We can  extend the definition of $\theta_i$ for $0\leq i \leq \ell-1$ for a simplified substitution to define $\theta_n$ where $n\in \Z$ as follows.  If $n\in \N$ and $(n)_\ell= n_{j}\cdots n_0\cdot$, we define 
$\theta_n:= \theta_{n_0} \circ \dots \circ \theta_{n_j}$. 
 Similarly, if $n\in -\N$  and  $(n)_\ell= \overline{(\ell -1)}\,n_{j}\cdots n_0 \cdot$, we define 
$\theta_n:= \theta_{n_0} \circ \cdots \circ \theta_{n_j} \circ \theta_{\ell -1}$.
We can also extend the notion of a column map applied to a sequence denoted by  a bold \label{emboldened_mu} 
 $\boldsymbol{\theta}_n:\mathcal{A}^{\mathbb{I}}\to\mathcal{A}^{\mathbb{I}}$ by setting  
\begin{equation}\label{emb_mu}
\boldsymbol{\theta}_n\left((u_k\right)_{k\in \mathbb{I}})=\left(\theta_n(u_k)\right)_{k\in \mathbb{I}} .\end{equation}

Given a substitution in simplified form, there is a very natural direct reading DFA that will generate a one-sided fixed point $\bfu$, given  by the proof of Cobham's theorem \cite{Cobham}.
Its initial state is the first letter $u_0$ of $\bfu$, its states are labelled by letters in $\mathcal A$, and its transition function is given by the maps $\theta_i$, as $\delta(a,i):= \theta_i(a)$. Moreover, although the classical literature concerning automata deals with one-sided sequences, it is straightforward to modify the statements to generate two-sided fixed points. Namely, to generate the entry $u_{-n}$ of the  left-infinite fixed point $\dots u_{-2}u_{-1}$, if $(n)_\ell= \overline{(\ell -1)}\,n_{j}\cdots  n_0\cdot$, we input the finite word $(\ell - 1)\,n_{j}\cdots  n_0$ in direct reading,
 but starting at the initial state $u_{-1}$. We summarise some of the information that we will need in the following lemma; its proof can be found in \cite{Cobham} for the case when $n\in \N_0$, and a simple generalisation works for $n\in -\N$. If ${\bfu}$ is a two-sided fixed point, we will call $u_{-1} \cdot u_0$ its {\em seed}.

 \begin{lemma}\label{lem:projection-0} Let $\theta$ be a length-$\ell$ substitution, in simplified form, with fixed point given by the seed $b\cdot a$. If $n\in \N_0$ with $(n)_{\ell} = n_k\cdots n_0\cdot$, then
 $u_n = \theta_{n_0} \circ\cdots  \circ\theta_{n_k} (a)$. If  $n\in -\N$ with  $(n)_{\ell} = (\overline {\ell -1 })\,n_k\cdots n_0\cdot$, then $u_n= \theta_{n_0} \circ\cdots  \circ\theta_{n_k} \circ \theta_{\ell -1} (b)$.
  \end{lemma}

\begin{example}\label{bigdiag_drex}

Define $\theta: \{a,b,c\}\rightarrow \{a,b,c\}^{3}$ as  $a\mapsto acb$, $b\mapsto baa$, and $c\mapsto bba$. While $\theta_0$ is an idempotent, the substitution  $\theta$ is not in  simplified form as $\theta_2$ is not an idempotent. However $\theta^2$ is in simplified form. The substitution $\theta^2$ is:\\\[a\mapsto acbbbabaa,\,b\mapsto baaacbacb,\,c\mapsto baabaaacb\] A minimal direct-reading automaton of $\theta^2$  that generates the fixed point given by the seed $b\cdot a$ is given in Figure \ref{pd_dir-pic}. We feed non-negative integers to the initial state $a$, and negative integers to the initial state $b$.

\end{example}

\begin{figure}[htbp]
\centering
\begin{tikzpicture}
\tikzset{->,>=stealth',shorten >=1pt,node distance=2.5cm,every state/.style={thick, fill=white}}
\node[state,initial,initial text=$\N_0$] (q0) {$a$};
\node[state, initial, above right of=q0,initial text=$-\N$] (q1) {$b$};
\node[state, below right of=q1] (q2) {$c$};
\draw 
(q0) edge[loop below] node{0,5,7,8} (q0)
(q0) edge[below] node{1} (q2)
(q0) edge[left] node{2,3,4,6} (q1)
(q1) edge[loop above] node{0,5,8} (q1)
(q1) edge[bend left, right] node{1,2,3,6} (q0)
(q1) edge[bend left, in=100, out=80, right] node{4,7} (q2)
(q2) edge[right] node{0,3,8} (q1)
(q2) edge[bend left, below] node{1,2,4,5,6} (q0)
(q2) edge[loop right] node{7} (q2)
;
\end{tikzpicture}
\caption{A minimal direct-reading automaton for $\theta^2$ in Example \ref{bigdiag_drex}}
\label{pd_dir-pic}
\end{figure}
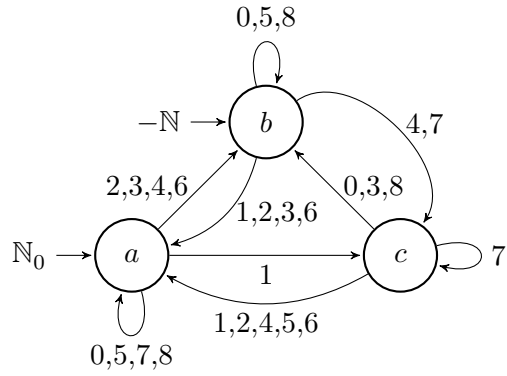

\subsection{The $\ell$-kernel and the reverse-reading automaton }
The \emph{$\ell$-kernel} of a sequence $\left(u_{n}\right)_{n\geq 0}$ is the collection of sequences
\[
\ker_\ell\bigl(\left(u_n\right)_{n \geq 0}\bigr) =
\bigl\{\bigl(u_{n\ell^{e}+j}\bigr)_{n\geq 0} \, : \, e\geq 0, \, 0\leq j\leq \ell^{e}-1\bigr\}.
\]

The following theorem is due to Eilenberg \cite{allouche-shallit-2003}. The {\em size} of an automaton is the cardinality of its state set.
\begin{theorem}[Eilenberg's theorem]\label{Eilenberg}
A sequence $\bfu$ is $\ell$-automatic in reverse reading if and only if
it has a finite $\ell$-kernel. Furthermore, the size  of a minimal finite-state automaton which generates $\bfu$ is the cardinality of $\ker_\ell(\bfu)$.
\end{theorem}

As with Cobham's theorem, Eilenberg's theorem is stated for one-sided sequences. However we can equally define the $\ell$-kernel of a two-sided sequence, and give a  statement and proof of Eilenberg's theorem that works for  two-sided sequences.

Let $\I=\N_0$ or $\I=\Z$. To study the $\ell$-kernel of a sequence $\bfu$, we introduce the operators $\Lambda_i: \mathcal A^{\I}\rightarrow \mathcal A^{\I}$, for $0 \leq i \leq \ell -1$ defined by
\[ \Lambda_i(\left(u_n\right)_{n\in \I}):= \left(u_{\ell n +i}\right)_{n\in \I}.\] 

In the case where the sequence $\bfu$ takes values in a finite field and $\ell=p^j$, with $p$ prime, and when $\I=\N$,  the operators $\Lambda_i$ are called the {\em Cartier operaters}, except that they act on formal power series instead of sequences.

The usual construction of a reverse-reading automaton that generates a one-sided automatic sequence, e.g., the one described in \cite{allouche-shallit-2003}, has states labelled with elements of $\ker_\ell\bigl(\bfu\bigr)$, and output obtained by projecting the 0-indexed entry,  $\pi_0 ((u_n)):=u_0$.   For $\bfv$, a sequence labelling one of the states, the transition function $\delta$ is $\delta(\bfv, i):= \Lambda_i(\bfv)$. We will give an alternate construction of the  reverse-reading automaton which will be useful for our application to Toeplitz sequences. To do this,
we link the operators $\Lambda_i$ to the column maps ${\bf \theta}_i$. Recall the definition of $\boldsymbol{\theta}_r$ in \eqref{emb_mu}.

\begin{proposition}\label{prop:column=cartier} Let $\bfu=\left(u_n\right)_{n\in \I}$ be a fixed point of a length-$\ell$ substitution $\theta$. Then for each $0\leq r\leq\ell-1$, we have
  \begin{equation}\label{lambda=mu eqn}
\Lambda_r(\bfu)=\boldsymbol{\theta}_r(\bfu)     . \end{equation}
\end{proposition}
\begin{proof}
Given a fixed point $\bfu$,  note that by definition, for each $k\in \I$, we have $u_{k\ell } \cdots u_{k\ell +\ell -1}= \theta(u_k)$.   We have $\theta(\bfu)=\bfu$ by definition. Therefore $ u_{k\ell +r}= \theta_r(u_k)$, so that   $\boldsymbol{\theta}_b  (\bfu)= ( \theta_r(u_k))_{k\in \I}= \left( u_{k\ell +r}\right)_{k\in \I} $. On the other hand, 
\[ \Lambda_r(\bfu)= \Lambda_r(\left(u_k\right)_{k\in \I}) =  \left(u_{k\ell+r}\right)_{k\in \I}, \]
from which the result follows.
\end{proof}

As with the definition of the column maps $\theta_n$, we can extend the defintion of $\Lambda_i$ for $0\leq i\leq \ell-1$ to the definition of $\Lambda_n$, where $n\in \N$ as follows: if $(n)_\ell = n_j\cdots n_0\cdot$, with $n_j \neq 0$, then 
$\Lambda_n:= \Lambda_{n_j} \circ \dots \circ \Lambda_{n_0}$. Similarly, if $n\in -\N$ and $(n)_\ell = \overline{(\ell -1)}n_j\cdots n_0\cdot$ with $n_j \neq \ell -1$, then 
$\Lambda_n:= \Lambda_{\ell -1}\circ \Lambda_{n_j} \circ \dots \circ \Lambda_{n_0}$.
We have the following as a corollary of Proposition \ref{prop:column=cartier}.

\begin{corollary}\label{cor:column=cartier}
Let $\bfu=\left(u_n\right)_{n\in \I}$ be a fixed point of a length-$\ell$ substitution $\theta$. Then for 
 any $n_0,n_1,\dots,n_d$ with $0\leq n_i\leq \ell-1$ for each $i$, we have
\begin{equation}\label{lam_mu_eqn}
\Lambda_{n_d}\circ\Lambda_{n_{d-1}}\circ\dots\circ\Lambda_{n_0}(\bfu)=
\boldsymbol{\theta}_{n_0}\circ\boldsymbol{\theta}_{n_1}
\circ\dots\circ\boldsymbol{\theta}_{n_d}(\bfu)
\end{equation}
\end{corollary}\begin{proof}
 Given $d\in \N$ and $0\leq n_0,n_1,\dots,n_d\leq \ell -1$, 
it can be shown as in the proof of Proposition \ref{prop:column=cartier} by induction that
\[ ({\boldsymbol \theta}_{n_0}\circ     {\boldsymbol \theta}_{n_{1}}\circ\dots\circ{\boldsymbol\theta}_{n_d}(\bfu))_n
=       u_{n \ell^{d+1} +n_d\ell^d +\dots+n_1\ell +n_0}\]
for each $n\in \I$.
We shall show that 
\begin{equation}\label{eq:ind-hyp}
\left(\Lambda_{n_d}\circ\dots\circ\Lambda_{n_1}\circ\Lambda_{n_0}( {u_n})\right)_{n\in \I}=\left(    u_{n\ell^{d+1} + n_d\ell^d +\dots+n_1\ell +n_0}        \right)_{n\in \I}\end{equation}
which will prove the corollary.

If $d=0$, the assertion follows by Proposition \ref{prop:column=cartier}. Inductively, assume that  \[\left(\Lambda_{n_{d-1}}\circ\dots\circ\Lambda_{n_1}\circ\Lambda_{n_0}( {u_n})\right)_{n\in \I}=\left(    u_{ n\ell^{d} +n_{d-1}\ell^{d-1} +\dots+n_1\ell +n_0}        \right)_{n\in \I}\]
holds for $d\in \N$.
   Then 
\begin{align*}
(\Lambda_{n_{d}}\circ\dots\circ\Lambda_{n_1}\circ\Lambda_{n_0}(
 {u_n}))_{n\in \I}
&=(\Lambda_{n_{d}}(\Lambda_{n_{d-1}}\circ\dots\circ\Lambda_{n_1}\circ\Lambda_{n_0}   
 ( {u_n})))_{n\in \I}
            \\
&=\Lambda_{n_{d}} \left(    u_{ n\ell^{d} +n_{d-1}\ell^{d-1} +\dots+n_1\ell +n_0}        \right)_{n\in \I}
\\
&=\left(u_{ n {\ell^{d+1}}  +  n_{d}\ell^{d}+ {n_{d-1}}{\ell^{d-1}}+\dots+{n_1}\ell +{n_0}}\right)_{n\in \I}
\end{align*}
and the induction is complete.
\end{proof}

Let  $\id: \mathcal A\rightarrow \mathcal A$ denote the identity map.

\begin{definition}
\label{def:structure-semigroup}
Let $\theta$ be a  length-$\ell$ substitution on $\mathcal A$ in simplified form. The {\em semigroup of $\theta$}, denoted $\mathcal S_\theta$ is the semigroup in $\mathcal A^\mathcal A$ defined by 
$$\mathcal S_\theta :=  \langle \id, {\theta}_i:0\leq i\leq \ell-1\rangle. $$

\end{definition}

We write elements  $s\in \mathcal A^\mathcal A$ as vectors, i.e., if  $\mathcal A:=\{ a_0, \dots , a_d\}$, then we write  $s= \begin{pmatrix} b_0, \dots ,b_d\end{pmatrix}^{T} $ for the function $s(a_i)= b_i$ for each $i$.

\begin{example}\label{pd intersection example}
Define $\theta: \{a,b\}\rightarrow \{a,b\}^{2}$ as  $a\mapsto ab$ and $b\mapsto aa$. This is the {\em period-doubling} substitution, and its simplified form is $\theta^2$; so $a\mapsto abaa$ and $b\mapsto abab$. We have 
 \[\mathcal S_{\theta^2}=\langle \id, {\theta^2}_i:0\leq i\leq 3\rangle = \{ \id,(a,a)^T,(b,b)^T\}.\]
 Note the importance of passing to a simplified form, otherwise the semigroup can be different, as it is here:
\[\mathcal S_{\theta^2} \subsetneq \mathcal S_\theta=  \langle \id, {\theta}_i:0\leq i\leq 1\rangle=\{ \id,(a,a)^T,(b,b)^T,(b,a)^T\}.\]
  \end{example}

A length-$\ell$ substitution  $\theta$ is {\em bijective} if every $\theta_j$ is a bijection, see \cite{Queffelec} for their study. If $\theta$ is bijective, then  $\mathcal S_\theta$ is a subgroup. Its importance is recognized in  \cite{Lemanczyk-Mentzen}, and it is related to the Ellis semigroup in \cite{Kellendonk-Yassawi}. Note that for non-bijective substitutions, the semigroup generated by the maps $\{ \theta_0, \dots , \theta_{\ell -1}\}$ does not necessarily contain $\id$.

The proof of the following lemma is a straightforward induction argument.
\begin{lemma}\label{lem:projection}
 If $\pi_0:\mathcal A^{\I}\rightarrow \mathcal A$ is the map which projects a sequence $(u_n)$ to its zero-indexed entry $u_0$, and if $n\in \N_0$ with $(n)_{\ell} = n_k \cdots n_0\cdot $,
then $u_n= \pi_0\left( \Lambda_{n_k}\circ\cdots\circ \Lambda_{n_0}({\bfu})\right)$. Similarly, if $\pi_{-1}:\mathcal A^{\Z}\rightarrow \mathcal A$ is defined as $\pi_{-1}((u_n)_{n\in Z}):= u_{-1}$, and $n\in -\N$ with  $(n)_{\ell} = \overline {(\ell -1) }\,n_k \cdots n_0\cdot$, then $u_n= \pi_{-1}\left( \Lambda_{\ell -1}\circ\Lambda_{n_k}\circ\cdots\circ \Lambda_{n_0}({\bfu})\right)$. 
\end{lemma}

\begin{theorem}\label{thm:automaton}
Let $\theta$ be a length-$\ell$ substitution on $\mathcal A$, in simplified form, and let $\mathcal S_\theta$  be the   structure semigroup of $\theta$. Then there is a minimal DFA $\mathcal M_\theta$, whose state set is $\mathcal S_\theta$, with the following property:
  for any fixed point  of $\theta$ with seed $ u_l\cdot u_r$,  there are output maps $\omega_l, \omega_r: \mathcal S_\theta \rightarrow \mathcal A$ where $(\mathcal M_\theta, \{\omega_l, \omega_r\})$ generates   ${\bfu}$ in reverse reading. 
\end{theorem}

\begin{proof}
Suppose that $\theta$ is defined on $\mathcal A=\{ a_0, \dots , a_d\}$.   We will  assume that we are given a bi-infinite fixed point, generated by 
$a_l \cdot a_r$.  The fact that $\theta$ is in simplified form implies that  $\theta(a_l)$ ends with $a_l$ and $\theta(a_r)$ starts with $a_r$.

We first define the finite-state automaton $\mathcal M=(\mathcal S_\theta, \Sigma, \delta, s_0, \mathcal A, \{\omega_r\})$ that generates the right part of ${\bfu}$. We will abuse notation and think of states as functions (technically they are  only labelled by functions). We let $\Sigma= \{ 0, \dots , \ell -1\}$,   and we let the initial state $s_0:=\id$. 
Given $a\in \mathcal A$, let $p_a:\mathcal S_\theta\rightarrow \mathcal A$ denote projection to the $a$-entry, $p_a(s):= s(a)$.

The output map $\omega_r: \mathcal S_\theta\rightarrow \mathcal A $ will be
$\omega_r=p_{a_r}.$ It remains to define the transition map $\delta$.

For $0\leq i \leq \ell -1$, define $\delta(\id, i): = \theta_i$.
Note that by definition any element  $s$ in $ \mathcal  S_\theta$ can be written as $ s =\theta_{n_0} \circ \dots \circ \theta_{n_k}$ for some $0\leq n_0, 
\dots ,n_k\leq \ell -1$. Now set  $\delta(s,i):= \theta_{n_0} \circ\cdots\circ  \theta_{n_k} \circ \theta_{i} (s_0)$; this is well-defined.

It remains to show that $\mathcal M$ generates the right-infinite part of ${\bfu}$. If $(n)_\ell = n_k \cdots n_0$, then
\[ u_n \stackrel{Lemma \,\, \ref{lem:projection}}{=} \pi_0 \left( \Lambda_{n_k}\circ\cdots \circ\Lambda_{n_0}({\bfu})\right)      \stackrel{Corollary  \,\,\ref{cor:column=cartier}}{=} \pi_0  \left( {\boldsymbol \theta}_{n_0}\circ\cdots \circ{\boldsymbol \theta}_{n_k}({\boldsymbol u})\right)    
\stackrel{Lemma\,\, \ref{lem:projection-0}}{=}  { \theta}_{n_0}\circ\cdots \circ{ \theta}_{n_k}( p_{a_r}(\id)), \]
and ${ \theta}_{n_0}\circ\cdots\circ { \theta}_{n_k}( p_{a_r}(\id)) =  p_{a_r}\left({ \theta}_{n_0}\circ\cdots\circ { \theta}_{n_k}(\id)\right)$ as $p_{a_r}$ commutes with each $\theta_i$. The result follows.

The left part of ${\boldsymbol u}$ will be generated similarly, except with the output map $\omega_{a_l}$. Furthermore, the states are in one-to-one correspondence with the kernel of $\theta$.
This can be seen using Lemma \ref{lem:projection}, which identifies the elements of the kernel as the images under compositions of the maps $\Lambda_i$. Hence a state is identified with the composition of the appropriate maps $\Lambda_i$ indexed by the edges in the path leading to that state. The initial state is identified with the fixed points.
 In particular, if a state is labelled by $s\in \mathcal S_\theta$, then it represents that sequence ${\bf s}( {\bfu})$. Thus, $\mathcal M$ is minimal by Eilenberg's theorem.
\end{proof}
\begin{definition}\label{def:automaton}
Let $\theta$ be a length-$\ell$ substitution on $\mathcal A$, in simplified form, and let $\mathcal S_\theta$  be the   structure semigroup of $\theta$. The  minimal DFA $\mathcal M_\theta$ defined in Theorem \ref{thm:automaton} is called the {\em semigroup automaton associated to $\theta$}.
\end{definition}

\begin{example}\label{bigdiag_rrex}
We look at  Example \ref{bigdiag_drex} again.  Recall that we have to consider $\eta:=\theta^2$ in order to work with a substitution in simplified form, if we want an automaton that generates bi-infinite fixed points. We start though with the simpler $\theta$, which is left simple, i.e., $\theta_0$ is an idempotent, so  we can construct with it an automaton that will generate right infinite $\theta$-fixed points; see Figure \ref{onesided}.

A minimal reverse-reading automaton for $\eta=\theta^2$, with states labelled using $\mathcal S_\theta$,  is  given in Figure \ref{fig:my_label}. To simplify the presentation of this automaton, Table  \ref{hello} lists the states $s_i\in \mathcal S_\theta$  in Figure \ref{fig:my_label}.

\begin{table}[ht]
\begin{center}
\renewcommand{\arraystretch}{1.4}
  \begin{tabular}{|c|c|c|c|c|}
   \hline
 $ s_0=\{a,b,c\}^T$ & $s_1=\{a,b,b\}^T$ & $s_2=\{c,a,a\}^T$ & $s_3=\{b,a,a\}^T$ & $s_4=\{b,a,b\}^T$ \\ 
    \hline
  $s_5=\{b,c,a\}^T$ &
  $s_6=\{a,b,a\}^T $& $s_7=\{a,c,c\}^T$ & $s_8=\{b,b,a\}^T $& $s_9=\{a,c,a\}^T$\\
    \hline
    $s_{10}=\{a,a,c\}^T $& $s_{11}=\{c,a,c\}^T$
& $s_{12}=\{a,a,b\}^T$ & $s_{13}=\{b,c,c\}^T$ & $s_{14}=\{c,b,b\}^T$\\
\hline
  $s_{15}=\{c,b,c\}^T $& $s_{16}=\{c,a,b\}^T$ & $s_{17}=\{b,c,b\}^T$
  &
  $s_{18}=\{c,c,a\}^T$ &$ s_{19}=\{c,c,b\}^T$\\
  \hline 
    $s_{20}=\{b,b,c\}^T$ & $s_{21}=\{a,a,a\}^T$ & $s_{22}=\{b,b,b\}^T$ & $s_{23}=\{c,c,c\}^T$ & \\
  \hline
  \end{tabular}\caption{The states in Figure \ref{fig:my_label}}\label{hello}
  \end{center}
\end{table}
\end{example}

\begin{figure}
    \centering
    \hspace*{-0.15in}
    \includegraphics[scale=0.7]{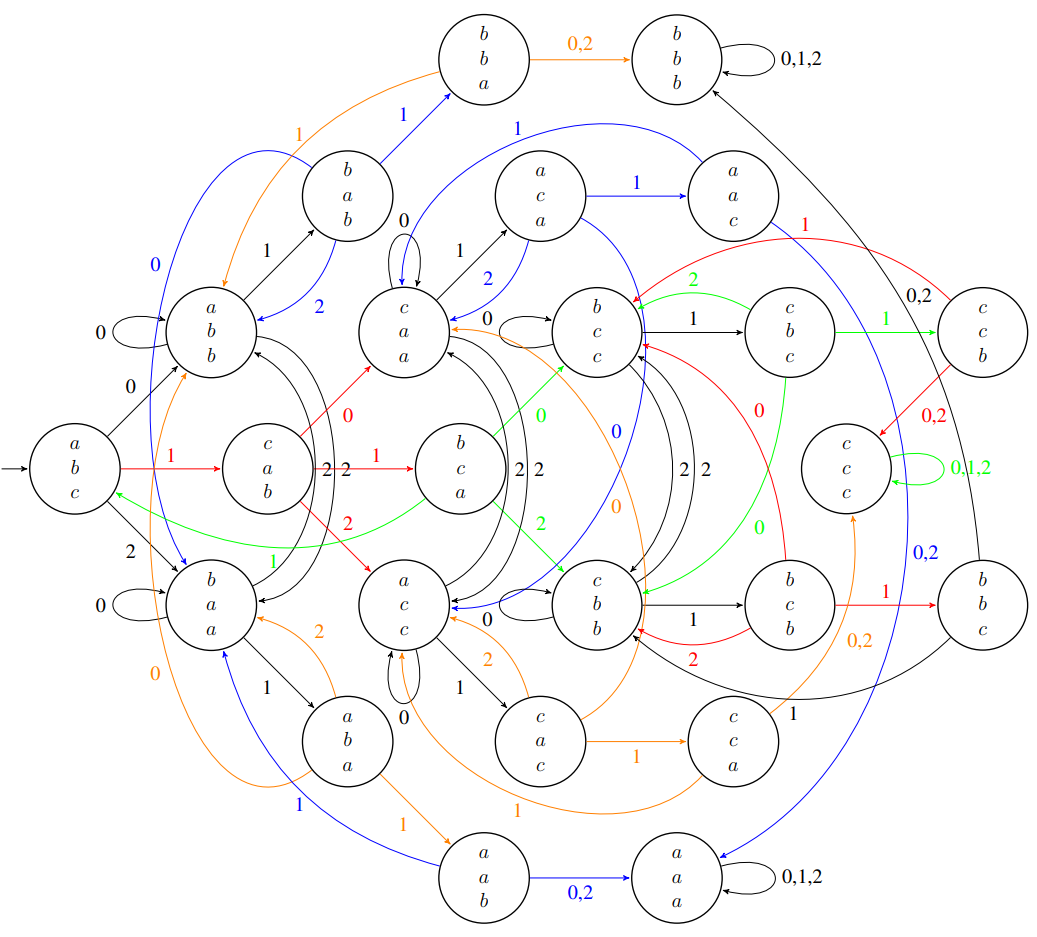}
    \caption{Reverse-reading automaton for Example \ref{bigdiag_rrex} generating a right-infinite one-sided automatic sequence; Any two intersecting, or close, edges  with the corresponding labels are coloured differently for more clarity.}
    \label{onesided}
\end{figure}

\begin{figure}
    \centering
    \hspace*{-0.2in}
    \includegraphics[scale=0.38]{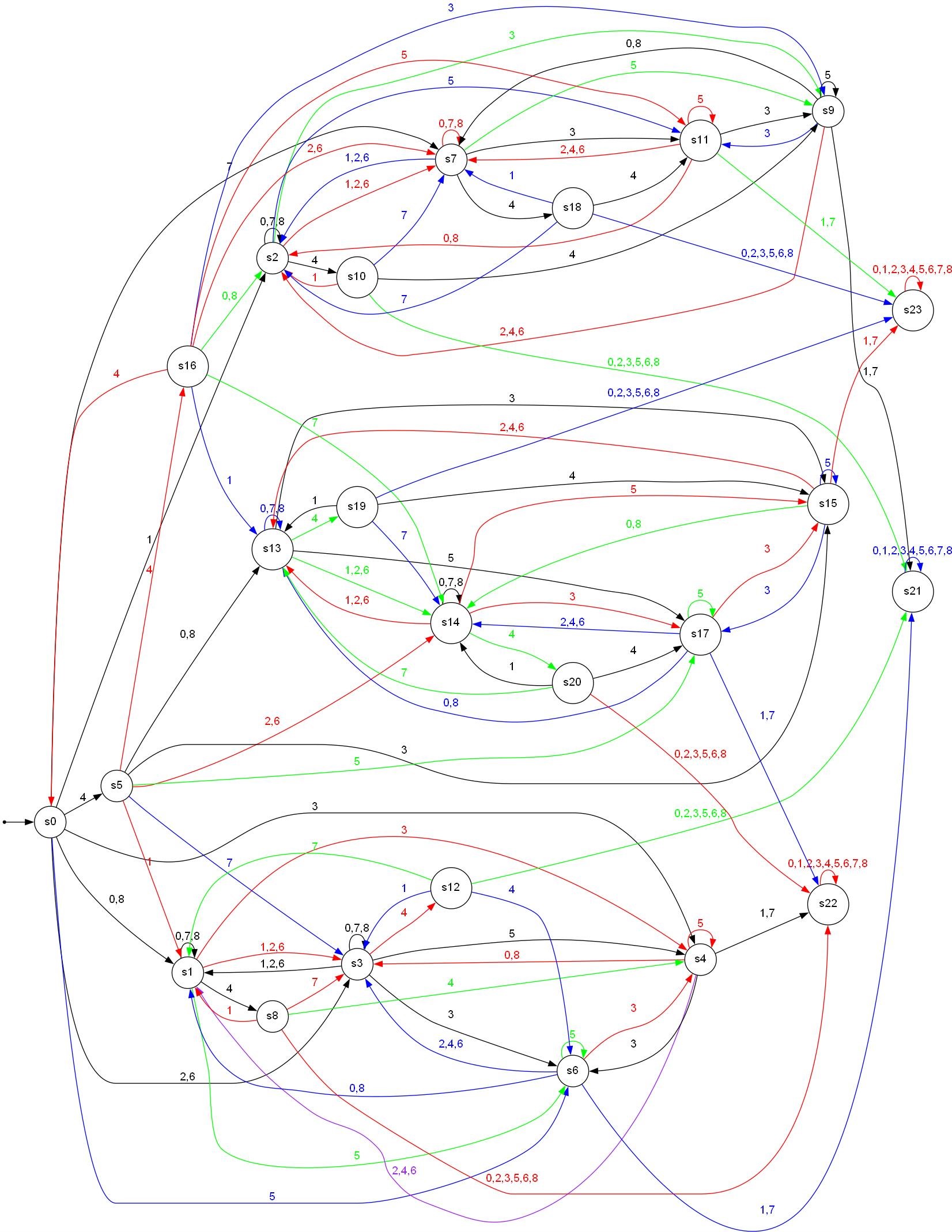}
    \caption{Reverse-reading automaton that generates bi-infinite fixed points for Example \ref{bigdiag_rrex}; Any two intersecting, or close, edges  with the corresponding labels are coloured differently for more clarity. Created by Graphviz web app.}
    \label{fig:my_label}
\end{figure}

\section{Toeplitz sequences and Toeplitz shifts}

There are many definitions of Toeplitz sequences, and a vast literature on Toeplitz shifts; we use the notions which bring us directly to our setting, which is that of substitution sequences.

   Let $\mathcal{A}$ be a finite, and let $\mathcal{A}^\mathbb{Z}$ denote the set of $\Z$-indexed infinite sequences over $\mathcal{A}$. 
 Endowing $\mathcal{A}$  with the discrete topology, we equip  $\mathcal{A}^\mathbb{Z}$ with the metrisable product topology.
 A {\em shift dynamical system}, or {\em shift},  is a pair $(X,\sigma)$ where $X$ is a closed $\sigma$-invariant set of $\mathcal A^\mathbb Z$ and $\sigma\colon \mathcal A^\mathbb Z\to \mathcal A^\mathbb Z$ is the left shift map  $(x_n)_{n\in\mathbb{Z}} \mapsto (x_{n+1})_{n\in\mathbb{Z}}$. 
     A  shift is {\em minimal} if it  has no non-trivial closed shift-invariant subsets.   
    We say that ${\bfx} \in \mathcal{A}^\mathbb{Z}$ is (shift-)\emph{periodic} if $\sigma^k({\bfx}) = {\bfx}$ for some $k \ge 1$, \emph{aperiodic} otherwise.
The~shift $(X,\sigma)$ is said to be aperiodic if each ${\bfx} \in X$ is aperiodic.   For basics on
  continuous and measurable dynamics see Walters \cite{Walters}. 

The simplest  to generate a shift is to  take a point ${\bfx}\in \mathcal{A}^\mathbb{Z}$ and $X$ to be the shift orbit closure of ${\bfx}$, i.e, $X:=\overline{\{ \sigma^n({\bfx}): n\in \Z\}}$.

If the  substitution $\theta$ is primitive, then the shift orbit closure of a $\theta$-periodic point is minimal and furthermore each $\theta$-periodic point generates the same shift space.  We write $(X_\theta,\sigma)$ to denote this shift, and 
we call it a \emph{substitution shift}. 
The substitution $\theta$ is \emph{aperiodic} if $(X_\theta,\sigma)$ is aperiodic. For details on the above and a study of  substitution shifts, see \cite{Queffelec}.

\subsection{The column number of a substitution}
Let $\theta$ be a primitive length-$\ell$ substitution with fixed 
point $\bfu$, and with $(X_\theta,\sigma)$  infinite. 
The {\em height }$h=h(\theta)$ of $\theta$   is defined as 
\[
h(\theta):= \max \{n\geq 1: \gcd(n,\ell)=1, n | \gcd\{a: u_a=u_0 \} \}\, .
\]
If $h>1$, this means that $\mathcal A$ decomposes into $h$ disjoint 
subsets: $\mathcal A_1\cup\dots\cup \mathcal A_h$, where a symbol from
$\mathcal A_i$ is always followed by a symbol from $\mathcal A_{i+1 \mod{h}}$ \cite{Dekking1977}.
The following theorem tells us that a length-$\ell$ substitution shift is a constant height suspension over another length-$\ell$ substitution shift which has trivial height $h=1$.

\begin{theorem}\label{thm:height} \cite[Remark 9, Lemmas 17 and 19]{Dekking1977}
If $\theta$  is a primitive aperiodic length-$\ell$ substitution defined on 
$\mathcal A$, and has height  $h$, then there is a length-$\ell$ substitution $ \bar\theta$ which  has trivial height, and  
$(X_{\theta}, \sigma) 
\cong(X_{  \bar  \theta}\times \{0,\dots,h-1\},T)$ where

\begin{equation*}
 T(x,i)\coloneqq 
\begin{cases}
  (x,i+1)      & \text{ if  } 
0\leq  i<h-1    \\
    (\bar \sigma(x), 0)   & \text{ if } i=h-1
\end{cases}
\end{equation*} 

\end{theorem}

The substitution $\bar\theta$ is
 called the {\em pure base} of $\theta$. If $\theta$ has trivial height $h=1$, then it is its own pure base.

Let $\theta$ have pure base $\bar\theta$.
We say that $\theta$ {\em has column number $c$}  if for some $k\in \N$, 
and some $(i_1, \dots ,i_k)$, 
$|\bar\theta_{i_1}\circ \cdots \circ\bar\theta_{i_k}(\mathcal A)|=c$
is the least such number. 
In particular, if $\theta$ has column number one, then we will say 
that  $\theta$ {\em has a coincidence}.

\subsection{Odometers and Toeplitz shifts}\label{sec:odometer}

Given a sequence $\left(\ell_n\right)_{n\geq 1}$ of natural numbers,
we work with the group \[\Z_{(\ell_n)} :=\prod_{n\geq 1} \Z/\ell_n\Z,\] where the group operation is given by coordinate-wise addition with carry.
For a detailed exposition of equivalent definitions of $\Z_{(\ell_n)}$, we refer the reader to \cite{Downarowicz2005}. 
Endowed with the product topology over the discrete topology on each $\Z/\ell_n\Z$, the group $\Z_{(\ell_n)}$ is a compact metrizable  topological group, 
where the unit $z=\bar{0}1\cdot$, which we simply write as $z=1$, is a topological generator.
We write elements $(z_n)$ of $\Z_{(\ell_n)}$ as left-infinite sequences $\dots z_2z_1\cdot $ where $z_n\in \Z/\ell_n\Z$,
so that addition in $\Z_{(\ell_n)}$ has the carries 
propagating to the left as usual in $\Z$. If $\ell_n=\ell$ is constant, then  $\Z_{(\ell_n)}= \Z_\ell$ is the classical ring of $\ell$-adic integers. Equivalently,
we can  define $\Z_{(\ell_n)}$   as the inverse limit     $\Z_{(\ell_n)}   =\varprojlim \Z/ (m_n \Z)$     of cyclic groups, where $m_n\coloneqq  \ell_1\cdots \ell_{n}$.

With the above notation, an {\em odometer} is a dynamical system $(\Z_{(\ell_n)}, +1)$.
 Odometers are  \emph{equicontinuous}, i.e., the family of function $\{+n : n\in\Z\}$ is equicontinuous. 

An equicontinuous factor of $(X,T)$ is {\em maximal} if any other equicontinuous factor of $(X,T)$ factors through it; the maximal equicontinuous factor always exists. 
A {\em  Toeplitz shift} is a symbolic shift $(X,\sigma)$, $X\subset \mathcal A^{\Z}$, 
which is a somewhere one-to-one extension of an odometer. Toeplitz shifts are always minimal. The following theorem tells us that the Toeplitz  shifts that are also length-$\ell$ substitution shifts are precisely those that  are primitive and  with column number one. The first two parts are due to Dekking \cite{Dekking1977}.

\begin{theorem} \label{thm:dekking}
Let $\theta$  be an aperiodic  primitive,  length-$\ell$ substitution of height $h$.
Then 
\begin{itemize}
\item
the maximal equicontinuous factor of $(X_\theta, \sigma)$ is the odometer $(\Z_\ell \times \Z/h\Z , (+1,+1))$, 
\item
 $(X_\theta,\sigma)$ is a Toeplitz shift if and only if $\theta$ has column number one, and
 \item
if $\theta$ has  column number one, then it has trivial height.
\end{itemize}
\end{theorem}

\begin{proof}
The first statement is proved in \cite[Theorem 13]{Dekking1977}. 

Let $\pi:     X_\theta \rightarrow  (\Z_\ell \times \Z/h\Z , (+1,+1))            $ be a maximal equicontinuous factor map. 
To prove the second statement, we will use  the fact that the minimal cardinality of  $\pi^{-1}(z)$, for $z\in  \Z_\ell \times \Z/h\Z$,  equals the column number of $\theta$ \cite[Theorem 3.1i]{Dekking1977}.
Suppose that  $(X_\theta,\sigma)$ is a Toeplitz shift, so that by definition, it is a somewhere one-to-one extension of an odometer. But then this odometer   {\em must} be the maximal equicontinuous factor of $(X_\theta,\sigma)$  \cite[Proposition 1.1]{Williams}. Thus $(X_\theta,\sigma)$ is a somewhere one-to-one extension of  $(\Z_\ell \times \Z/h\Z , (+1,+1))$, which implies that $\theta$ has column number one.
Conversely, if $\theta$ has column number one, then $\pi$ is somewhere one-to-one so that  $(X_\theta,\sigma)$ is Toeplitz.

The third statement is proved by Lemanzcyk and Muellner in  \cite[Lemma 2.2]{Lemanczyk-Mullner}.

\end{proof}

\subsection{Toeplitz sequences}\label{Toeplitz-sequence} A {\em Toeplitz sequence} is a  sequence that is obtained by filling its entries one arithmetic sequence of indices at a time  with a constant symbol, as follows.
Let $\mathbb I= \mathbb N_0$  or $\mathbb Z$.  Let $\mathcal A$ be a finite alphabet. For $\bfx\in \mathcal A^{\mathbb I}$, $k\in \N$ and $a\in \mathcal A$, we define recursively
\[ \Per_k(\bfx,a):= \{ n\in \mathbb I: x_{n'} = a \mbox{ for all } n'\equiv n \bmod k   \mbox{ and } n \not\in \Per_j(\bfx,b) \mbox{ for any } b\in \mathcal A \mbox{ and } j|k   \},\]
\[ \Per_k(\bfx) := \bigcup_{a \in\mathcal A}  \Per_k(\bfx,a),\]
\[ \Per(\bfx) := \bigcup_{k}  \Per_k(\bfx),\]
and 
\[ \Aper(\bfx):=\mathbb I \backslash       \Per(\bfx)     .\]
We say that the sequence $\bfx$ is  {\em Toeplitz} if   $\Aper(\bfx)=\emptyset$.   If  $\bfx$ is Toeplitz and  $\{k_n: n\in \N\}$ is the set of integers such that $\Per_k(\bfx)\neq \emptyset$, we call $\{k_n: n\in \N\}$ the {\em set of essential periods of $\bfx$}. A set of essential periods can be turned into a {\em period structure} $\{m_n: n\in \N\}$, where each $m_n$  is an essential period and $m_n $ divides $m_{n+1}$; see \cite[Proposition 2.1]{Williams}.
  The following result gives the connection between Toeplitz sequences and Toeplitz shifts. Namely,  a Toeplitz shift always contains  Toeplitz sequences, and a Toeplitz sequence generates a Toeplitz shift.  It dates back to   Williams, for details see \cite[Theorem 2.2, Corollary 2.4]{Williams}; see also \cite[Theorem 5.1]{Downarowicz2005}.
  \begin{theorem}\label{thm:Toeplitz}
 If ${\bfx}$ is a Toeplitz sequence, and $X$ is its shift orbit closure, then $(X,\sigma)$ is a Toeplitz shift whose maximal equicontinuous factor is the odometer $(\varprojlim \Z_{(m_n)},+1)$ generated by the essential period structure of ${\bfx}$.   Conversely, if $(X,\sigma)$ is a Toeplitz shift, with the somewhere injective equicontinuous factor $\pi :X\rightarrow \Z_{(\ell_n)} $,  then any point ${\bfx}$  such that $\{{\bfx}\}= \pi^{-1}(\pi({\bfx}))$ is a Toeplitz sequence, and its period structure generates $\varprojlim \Z_{(m_n)} $.  
  \end{theorem}
   \begin{definition}\label{def:semicocycle}
 The {\em semicocycle} defined by $\left(x_n\right)_{n\in \I}$ is the map $n\in \I \mapsto x_n \in \mathcal A$.   \end{definition}

  To discuss the discontinuities of a semicocycle, we briefly describe the topology that we impose on $\mathcal A$ and   $\I$. The finite set  $\mathcal A$ has the discrete topology.
 By   Theorem
	 \ref{thm:dekking}, a Toeplitz shift $(X,\sigma)$ has the odometer $(\Z_\ell, +1)$ as maximal equicontinuous factor. The set $\I$ inherits a topology from $\Z_\ell$.  For, $\I$ is included in $\Z_\ell$ via the map which sends $n$ to its base-$\ell$ expansion, as described in Section \ref{sec:preliminaries}. From this inclusion $\I\subset \Z_\ell$,   the integers  $m,n\in \I$ are close if and only if their base-$\ell$ expansions agree on a large initial block.
	 
	Next we will  see that this choice of topology implies that semicocycle discontinuities are linked to arithmetic progressions of common difference $\ell^j$.
We write such a progression as 
\[\boldsymbol{a}_{(i,j)}:=(k\ell^j +i )_{k\in\I}\]
 where $ |i|\leq \ell^j-1$.
We say that a sequence $\bfx$ is constant on $\boldsymbol{a}_{(i,j)}$ if there is an $ a$ such that $x_n=a$ whenever $n\in\boldsymbol{a}_{(i,j)}$.
\begin{lemma}\label{lem:semicocycle}
Let $f: \I \to {\mathcal A}$ be the semicocycle defined by $\bfx$. Then $f$ is discontinuous at n if and only if for all $j$, there exists an arithmetic progression $\boldsymbol{a}_{(i,j)}$ of common difference $\ell^j$ such that $n\in \boldsymbol{a}_{(i,j)}$ and $\bfx$ is not constant on $\boldsymbol{a}_{(i,j)}$. 
  \end{lemma}
     \begin{proof}
First  suppose that  $f$ is discontinuous at $n$. This implies that for all $j$, there exists $m$ such that $(m)_\ell$ and $(n)_\ell$ agree on the $j$ least significant  entries, but $x_m \neq x_n$. In other words, there is $| i| <\ell^j$ such that 
 \[ n=k\ell^j +i \text{ and }m=k'\ell^j +i;\]
  so that  $n,m\in \boldsymbol{a}_{(i,j)}$, and $\bfx$ is not constant on $\boldsymbol{a}_{(i,j)}$.
 
 Conversely, suppose that there is an $n$  such that for each $ j$ there is an $ \boldsymbol{a}_{(i,j)}$  with $n\in \boldsymbol{a}_{(i,j)}$, but where there is  $m\in \boldsymbol{a}_{(i,j)} $ with 
 $x_m\neq x_n$.
 Since $m,n\in \boldsymbol{a}_{(i,j)}$, we have  $d(n,m)<1/{\ell^j}$. Since such $m$ can be found for any $j$,  we have shown that the semicocycle $f$ is discontinuous at $n$.
\end{proof}

Note that a  substitution shift may be Toeplitz, but its fixed points are not necessarily  Toeplitz sequences,  as  Example \ref{pd_ex} shows.

\begin{example}\label{pd_ex}
Recall the period-doubling substitution $\theta$ from Example \ref{pd intersection example}. 
We will show that the unique right-infinite fixed point for $\theta$ is a Toeplitz sequence, but neither of the bi-infinite fixed points for $\theta^2$ are Toeplitz.
A minimal reverse-reading automaton that generates the period-doubling sequence $\bfx$ with states as $\theta_i\in\mathcal A^{\mathcal A}$ is the following:
\begin{figure}[H]
\centering
\begin{tikzpicture}
\tikzset{->,>=stealth',shorten >=1pt,node distance=2.5cm,every state/.style={thick, fill=white},initial text=$ $}
\node[state,initial] (q0) {$\begin{matrix}a\\b\end{matrix}$};
\node[state, above of=q0] (q1) {$\begin{matrix}a\\a\end{matrix}$};
\node[state, right of=q0] (q2) {$\begin{matrix}b\\b\end{matrix}$};
\draw 
(q0) edge[left] node{0,2} (q1)
(q0) edge[below] node{1} (q2)
(q0) edge[loop below] node{3} (q0)
(q1) edge[loop right] node{0,1,2,3} (q1)
(q2) edge[loop right] node{0,1,2,3} (q2)
;
\end{tikzpicture}
\caption{A minimal reverse-reading automaton generating $\bfx$}
\label{period_doubling_machine}
\end{figure}
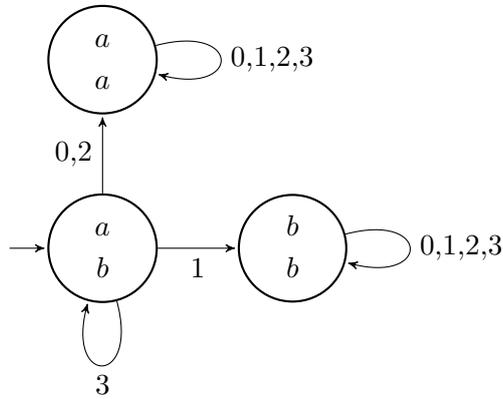

\noindent Now, we generate the fixed point of $\theta^2$ using the algorithm used to define the Toeplitz sequence. To generate the sequence $\bfx\in\mathcal A^\Z$, we take $k=2^i$ for each iteration of choice of an arithmetic sequence; i.e., $i=1,2,\dots$ and fill the letters $a$ and $b$ on alternate iterations starting with $a$. For $i=1$, $\Per_2(\bfx,a)=2\I$. This is because $\theta_0 = (a,a)^{T}$. Next,   $i=2\implies\Per_4(\bfx,b)=4\I+1$, because 
$(\theta^{2})_1=(b,b)^T$. Similarly,  $i=3\implies\Per_8(\bfx,a)=8\I+3$; $i=4\implies\Per_{16}(\bfx,b)=16\I+7$, and so on, so that $\Per_k(\bfx)=\I \backslash \{-1\}$. So, $\Aper(\bfx)=\{-1\}\neq\emptyset$. Another way to state this is that every entry in the two fixed points possible is the same except  at $x_{-1}$. Hence, the bi-infinite fixed points are not Toeplitz sequences.  Although, notice that the right-infinite fixed point of $\theta$ is in fact a Toeplitz sequence, as every index in $\N_0$ belongs to an arithmetic progression. 
The column number of the substitution $\theta$ is one, which confirms that it generates a Toeplitz shift. \end{example}

 Recall that  each state $s$ in the semigroup  automaton $\mathcal M$ is labelled by  a map  $f_s:\mathcal A \rightarrow \mathcal A$. 
 Call the state $s$ a {\em $k$-vertex} if 
  $|k| = |\Image f_s|$. If $s$ is a $1$-vertex, then $f_s$ is the function which is  projection to the letter $a$.

\begin{definition}[Reduced graph of a substitution]
Let $\theta$ have a coincidence. If we remove from its semigroup  automaton all $1$-vertices, and all edges leading to them, we are left with the {\em reduced graph} of $\theta$.
\end{definition}

From Example \ref{pd_ex}, the reduced graph of $\theta^2$ would be the following:
\begin{figure}[htbp]
\centering
\begin{tikzpicture}
\tikzset{->,>=stealth',shorten >=1pt,node distance=2.5cm,every state/.style={thick, fill=white},initial text=$ $}
\node[state,initial] (q0) {$\begin{matrix}a\\b\end{matrix}$};
\draw 
(q0) edge[loop below] node{3} (q0)
;
\end{tikzpicture}
\caption{The reduced graph of $\theta^2$}
\label{period_doubling_machine}
\end{figure}
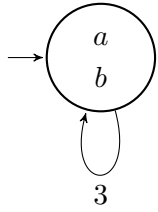

According to Theorem \ref{thm:main-result}, this confirms again that the aperiodic part of the sequence is $-1$, as the infinite loop $\bar{3}\cdot$ is the base-$4$ representation of $-1$.

In the following theorem we use $\pi_a$ to denote the element in $\mathcal A^{\mathcal A}$ which is the projection function $x\in \mathcal A \mapsto a$, and 
given $a\in \mathcal A$, let $p_a:\mathcal A^{\mathcal A}\rightarrow \mathcal A$ denote projection to the $a$-entry, $p_a(s):= s(a)$.

\begin{theorem}\label{thm:main-result}
Let $\theta$ be a primitive aperiodic length-$\ell$ substitution in simplified form,  and let  $\mathcal M_\theta$ be the semigroup  automaton associated to $\theta$.     Let   $\bfu$ be a fixed point, generated by $(\mathcal M_\theta, \{\omega_l, \omega_r\})$.  If $(X_\theta, \sigma)$ is Toeplitz, then 
\[\Per(\bfu) = \{n\in \mathbb Z: \delta(\id, (n)_\ell)= \pi_{u_n} \}\]
Hence the reduced graph of $\theta$ is a description of the set of semicocycle discontinuities of ${\bf u}$.      
\end{theorem}

\begin{proof}

We suppose, wlog,  that $n\in\N_0$ and that $n\in \Per(\bfu)$. By Theorem  \ref{thm:dekking}, $\theta$ has trivial height, and 
$(\Z_\ell, +1)$ is the maximal equicontinuous factor of $(X_\theta, \sigma)$. Next Theorem \ref{thm:Toeplitz} tells us that 
the set of essential periods is $\{\ell^j: j\in \N\}$. Thus if $n\in \Per(\bfu)$, then 
   for some  $j$ with $n<\ell^j$, $(u_{k\ell^j+n})_{k\in \mathbb Z} $   is a  constant sequence.
       In what follows we can pad  $(n)_\ell$ with some leading zeros if necessary so that the length of  $(n)_\ell$ equals $j$.  Note that this padding will not change the action of our automaton on fixed points. For, since $\theta_0$ is an idempotent, we have that for each $0\leq l \leq\ell -1$, 
$p_{u_0}(\theta_l\circ \theta_0^m (\id))=p_{u_0}( \theta_l\circ \theta_0(\id))$.

Writing $(n)_\ell= n_{j-1}\dots n_0 \cdot$,  since $(u_{k\ell^j+n})_{k\in \mathbb Z}$ is a constant sequence, then $\Lambda_{n_{j-1}} \circ \dots  \circ \Lambda_{n_0} (\bfu)$ is constantly equal to $a=u_n$. 
By definition of $\mathcal M_\theta$, this means that $p_{u_0} (\delta(\id,  (n)_\ell))= u_n$. We claim that  $\delta(\id,  (n)_\ell)= \pi_{u_n}$. Suppose not. Then there is a letter $v$ such that $p_v(  \delta(\id,  (n)_\ell))= b \neq u_n$.
Using the fact that $\theta$ is primitive, we know that $\theta^j(u_0)$ contains an occurrence of $v$; we will assume that  $j=1$ (otherwise we have to work with a larger $k$ in what follows, whose base-$\ell$ expansion has length $j$.)
Choose $i$ such that $\theta_i(u_0)=v$. Since our assumption is that $(u_{k\ell^j+n})_{k\in \mathbb Z} $ equals the constant sequence $\overline{u_n}$, then letting $k=i$, and noting that $(i\ell^j+n)_\ell= i(n)_\ell$, we have
\begin{align*}
u_n=u_{i\ell^j+n}= p_{u_0}( \delta(\id,  i(n)_\ell)). 
\end{align*}
On the other hand
\begin{align*}   p_{u_0}( \delta(\id,  i(n)_\ell)) =           p_{u_0}(      \theta_{n_0} \circ \dots \circ  \theta_{n_{j-1}}\circ  \theta_i (\id) ) &=      \theta_{n_0} \circ \dots \circ  \theta_{n_{j-1}}\circ p_{u_0} \theta_i (\id)\\& =    \theta_{n_0} \circ \dots \circ  \theta_{n_{j-1}}(v) \\& =  \theta_{n_0} \circ \dots \circ  \theta_{n_{j-1}} \circ p_v(\id) \\&=p_v ( \theta_{n_0} \circ \dots \circ  \theta_{n_{j-1}} (\id)) = b, \end{align*} 
a contradicition to our assumption that $b\neq u_n$. Therefore  $\delta(\id,  (n)_\ell)= \pi_{u_n}$.

The case when $n\in -\N$ is similar: we just explain how to proceed with $(n)_\ell$.  Recall that if  the natural number $-n$ has a base-$\ell$ expansion of length $j'$, then $(n)_\ell= \overline{(\ell -1)}(\ell^{j'} -n)_\ell$.  As we assume that the substitution is in simplified form, we can replace $(n)_\ell$ by $ (\ell -1)(\ell^{j'} -n)_\ell$, a word of length $j'+1$.
If  for some  $j$ with $|n|<\ell^j$, $(u_{k\ell^j+n})_{k\in \mathbb Z} $ is  a  constant sequence, we can assume, by padding $ (\ell -1)(\ell^{j'} -n)_\ell$ with extra copies of $\ell -1$, to obtain a word of length $j$ that will represent $(n)_\ell$. Now the rest of the proof is the same, except that we work with $p_{u_{-1}}$ instead of $p_{u_0}$.

Conversely, suppose that $\delta(\id, (n)_\ell)= \pi_{u_n}$.  Since  $\theta$ is in simplified form, we can assume that all base-$\ell$ expansions of integers are finite, by cropping off all but one most significant entry $\ell -1$ in the expansion of  a negative integer. Suppose that $(n)_\ell$ has length $j$. We will show that $ (u_{k\ell^j+n})_{k\in \mathbb Z}$ is constant.  
If  $k\in \Z$  then
\begin{align*}
u_{k\ell^j+n } = p_{u_0}(\delta( id, (k\ell^j+n)_\ell)) =  p_{u_0}(\delta(   \delta( id, (n)_\ell),(k)_\ell)) =  p_{u_0}(\delta( \pi_{u_n},(k)_\ell))=   p_{u_0}(\pi_{u_n}) = u_n,
\end{align*}
and the result follows.

\end{proof}
 
\bibliography{gandhar-reem}{}
\bibliographystyle{amsplain}

	\end{document}